\documentclass{article}
\usepackage{amsmath,amssymb,amsthm,amsfonts,mathtools,mathdots,nicefrac,tikz,enumitem,url,hyperref}
\usepackage{multirow}
\usepackage[a4paper,top=2.5cm,bottom=2cm,left=2cm,right=2cm,marginparwidth=1.75cm]{geometry}

\usepackage{algorithm}
\usepackage[noend]{algpseudocode}
\theoremstyle{definition} 
\newtheorem{theorem}{Theorem}[section]
\newtheorem{corollary}{Corollary}[section]
\newtheorem{lemma}{Lemma}[section]
\newtheorem{proposition}{Proposition}[section]

\newtheorem{example}{Example}[section]
\newtheorem{remark}{Remark}[section]
\newtheorem{conjecture}{Conjecture}[section]

\newcommand{\sn}{\mathrm{sn}}

\DeclareMathOperator{\val}{val}

\DeclareMathOperator{\gon}{gon}

\DeclareMathOperator{\scw}{scw}

\DeclareMathOperator{\lcm}{lcm}

\usepackage{mathrsfs}

\algdef{SE}[SUBALG]{Indent}{EndIndent}{}{\algorithmicend\ }%
\algtext*{Indent}
\algtext*{EndIndent}

\title{Chip-Firing Games on Banana Trees}
\author{
  Marchelle Beougher, Nila Cibu, Kexin Ding, Steven DiSilvio, Kristin Heysse, \and Sasha Kononova, Chan Lee, Ralph Morrison, Krish Singal
}
\date{} 

\begin{document}

\maketitle

\begin{abstract}
    We study chip-firing games on multigraphs whose underlying simple graphs are trees, paths, and stars, denoted as banana trees, paths, and stars respectively. We present a polynomial time algorithm to compute the divisorial gonality of banana paths, and give combinatorial formulas for the related invariants of scramble number and screewidth for any banana tree.  Furthermore, we leverage banana paths to show that gonality can increase or decrease by an arbitrary amount upon deletion of a single edge, even when the resulting graph is connected. Lastly, we study banana trees and Brill-Noether theory to prove that the gonality conjecture holds for all banana trees. 
\end{abstract}

\section{Introduction}

 As pioneered by Baker and Norine \cite{bn}, chip-firing games on graphs provide a discrete analog of divisor theory on algebraic curves.  For instance, graph gonality serves as a combinatorial tool for studying the gonality of algebraic curves, which in both contexts be defined as the minimum degree of a positive rank divisor. Accordingly, computing the gonalities of various families of graphs has provided algebraic geometers with more graph-theoretic tools for finding the gonalities of associated algebraic curves. To this end, we focus on a particularly versatile family of graphs called \emph{banana trees}.

\begin{figure}[hbt]
    \centering
\includegraphics{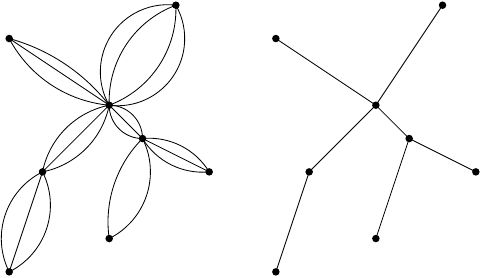}
    \caption{A banana tree on the left, and its underlying simple graph on the right}
    \label{figure:banana_tree_example}
\end{figure}

Banana trees are multigraphs whose underlying simple graph is a tree. An example of a banana tree, along with its underlying simple graph, is illustrated in Figure \ref{figure:banana_tree_example}. In the case that the underlying tree is a path, we call the multigraph a \emph{banana path}; these have previously been referred to to as generalized banana graphs. They form a structurally simple family of graphs with surprisingly rich properties concerning divisor theory on graphs. Banana paths were leveraged in \cite{gonseq} to completely determine for which integers \(m\) and \(n\) there exists a graph \(G\) with first gonality \(m\) and second gonality \(n\). Later work in \cite{semigroup_gonality_sequences} used banana paths to push these results further by providing a partial classification for achievable first, second, and third gonalities. These applications of banana trees are a testament to their versatility and utility in divisor theory, warranting a deeper study of these graphs.

Our first major result is a polynomial-time algorithm for computing the gonality of a banana path
\begin{theorem}\label{theorem:polytime_gonality} The gonality of a banana path on $n+1$ vertices can be computed in $O(n^3)$ time.
\end{theorem}
Our proof relies on a recursion for a slight generalization of gonality combined with dynamic programming.
This result is in contrast to computing gonality for general graphs, which is known to be an NP-complete problem~\cite{gijswijt2019computing}.

Our next major result leverages banana paths to study how gonality can behave under taking subgraphs.  It has already been noted that gonality is not subgraph monotone:  it is possible for gonality to increase by arbitrarily large amounts when taking subgraphs \cite{sparse}.  We strengthen this result by proving that gonality can increase by any precise amount upon the deletion of a single edge.

\begin{theorem}\label{theorem:peeling_edges}
    For every integer \(r\geq 1\), there exists a connected graph $G$ with an edge $e$ such that deleting $e$ preserves connectivity and increases the gonality of $G$ by exactly \(r\):
    \[\gon(G-e)=\gon(G)+r.\]
\end{theorem}
We also consider banana trees from the perspective of Brill-Noether theory.  One of the longest-standing open conjectures regarding divisor theory on graphs is the Brill-Noether conjecture, which in the rank one case states that gonality of a graph is can be upper bounded by a linear function of the graph's genus \(g\), where \(g = |E| - |V| + 1\). Namely:
\[\gon(G)\leq \left\lfloor \frac{g+3}{2}\right\rfloor\]
Recent work in \cite{discrete_metric_different} has provided evidence against this conjecture, so it is natural to search for counterexamples.  Sadly, banana trees will not provide them.

\begin{theorem}\label{theorem:brill_noether_bound} 
    For any banana tree \(G\) of genus \(g\), we have
    \[\gon(G)\leq \left\lfloor \frac{g+3}{2}\right\rfloor,\]
    with a necessary condition for equality being that all edge bunches have size at most $4$.  Moreover, only finitely many bridgeless banana trees achieve equality.
\end{theorem}

Our paper is organized as follows. In Section \ref{section:background}, we provide general background and definitions of chip-firing games and banana paths. In Section \ref{section:algorithms}, we present an algorithm for computing the gonality of banana trees, as well as results on the related invariants of scramble number and screewidth of banana trees. We also show specific results regarding monotone and unimodal graphs, we prove Theorem \ref{theorem:polytime_gonality}. In Section \ref{section:monotoneandunimodal}, we find lower bounds on certain banana path gonalities and prove Theorem \ref{theorem:peeling_edges}. In Section \ref{section:brill} we consider banana trees from the perspective of Brill-Noether theory.  Lastly, in Section \ref{section:computationalresults}, we provide some computational results.

\medskip

\noindent \textbf{Acknowledgements.} The authors were supported by Williams College and by the NSF via grants DMS-2241623 and DMS-1947438.  The third author was supported by Brown University.

\section{Background and definitions} \label{section:background}
A graph $G$ is a pair \(G=(V,E)\), where $V$ is a finite set of vertices and $E$ is a finite multiset of edges, defined as unordered pairs of vertices in $V$. If $E$ is a set (that is, if no edge is repeated), we call $G$ a \emph{simple graph}. Throughout this paper, we assume our graphs are connected, undirected, and loopless. The \emph{underlying simple graph of $G$} is the graph obtained by deleting all but one edge between any two adjacent vertices in $G$.  Given a vertex \(v\in V(G)\), the \emph{valence of $v$}, denoted \(\val(v)\), is the number of edges incident to \(v\).  Given a set of vertices $S\subset V(G)$, the \emph{subgraph induced by $S$}, denoted $G[S]$, is the graph with vertex set $S$ and edge set consisting of those edges in $E(G)$ with both endpoints in $S$.

A \emph{divisor} \(D:V(G)\rightarrow \mathbb{Z}\) on a graph \(G\) is an assignment of integers to the vertices of \(G\). We intuitively think of this assignment as a placement of poker chips on the vertices of \(G\), with $D(v)$ chips placed on $v$, with negative integers representing debt. The set \(\textrm{Div}(G)\) of all divisors on \(G\) forms a group under vertex-wise addition; in particular, it is the free Abelian group on \(V(G)\).  We write a divisor \(D\) as
\[D=\sum_{v\in V(G)}D(v)\cdot v.\]
Sometimes it is useful to refer to the number of chips \(D\) places on \(v\) as \(D(v)\).
The \emph{degree} of a divisor is the sum of the integers:
\[\deg(D)=\sum_{v\in V(G)}D(v).\]
We say a divisor \(D\) is \emph{effective} if \(D(v)\geq 0\) for all \(v\), i.e. if no vertex is in debt.

We can transform one divisor into another through \emph{chip-firing moves}.  Given a divisor \(D\) and a vertex \(v\in V(G)\), \emph{chip-firing at $v$} transforms \(D\) into \(D'\) by subtracting \(\val(v)\) chips from \(v\) and adding chips to the neighbors \(w\) of \(v\), one chip for each edge connecting \(v\) to \(w\).  Formulaically,
\[D'=(D(v)-\val(v))\cdot v+\sum_{w\in V(G)-\{v\}}(D(w)+|E(v,w)|)\cdot w.\]
We say that two divisors \(D\) and \(D'\) are \emph{linearly equivalent}, denoted \(D\sim D'\), if one can obtain \(D'\) from \(D\) through a sequence of chip-firing moves.  For example, in Figure \ref{figure:first_banana_path_example} we have three equivalent divisors on a graph \(G\); the second is obtained from the first by firing the middle vertex $v_1$, and the third from the second by firing the leftmost vertex $v_0$.  The first and third divisors are effective; the second is not.

\begin{figure}[hbt]
    \centering
    \includegraphics[scale=1.3]{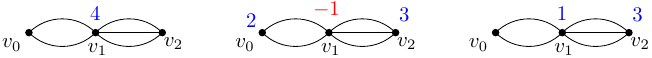}
    \caption{The resulting divisors from from firing the $v_1$ followed by firing $v_0$.}
    \label{figure:first_banana_path_example}
\end{figure}

Since the order in which vertices are fired does not matter, it is often useful to think of firing multiple vertices simultaneously.  Given a subset \(S\subset V(G)\), the \emph{subset-firing move at \(S\)} fires every vertex of \(S\).  The net effect of such a move is that one chip moves along every edge connecting \(S\) to \(S^C\), and no other chip movement occurs (since any two adjacent vertices in \(S\) will send chips to one another and cancel out).  For instance, firing the set $S=\{v_0,v_1\}$ in Figure \ref{figure:first_banana_path_example} yields the third divisor.  Note that although the individual chip-firing moves introduced intermediate debt, this did not occur for this particular subset-firing move.  If there is no debt in a divisor before or after a subset-firing move, we refer to that subset-firing move as \emph{legal}.  It turns out that when considering equivalence of divisors, it suffices to consider legal subset-firing moves:

\begin{lemma}\label{lemma:level-set}[Corollary 3.11 in \cite{db}]
Let $D$ and $D'$ be equivalent effective divisors.  There exists a finite sequence of subset-firing moves that transforms $D$ into $D'$ without ever going into debt.
\end{lemma}

For any divisor $D$ and any vertex $q$, there exists a unique $q$-reduced divisor $D_q$ equivalent to $D$.  This is a divisor such that
\begin{itemize}
    \item[(i)] $D(v)\geq 0$ for all $v\neq q$, and
    \item[(ii)] There does not exist a legal subset-firing move $S\subset V(G)-\{q\}$.
\end{itemize}
For instance, for the divisors illustrated in Figure \ref{figure:first_banana_path_example}, the first two divisors are not $v_2$-reduced; the first has a legal firing set $\{v_0,v_1\}$, and the second has a vertex besides $v_2$ in debt.  The third however is $v_2$-reduced.  We refer the reader to \cite[Ch. 3]{sandpiles} for more details.

We now define the \emph{rank} of a divisor, denoted \(r(D)\).  If \(D\) is not equivalent to any effective divisor, we set \(r(D)=-1\).  Otherwise, \(r(D)\) is the largest nonnegative integer \(r\) such that for every effective divisor \(E\) with \(\deg(E)=r\), we have that \(D-E\) is equivalent to an effective divisor.  Intuitively, \(r(D)\) is the maximum amount of added debt that \(D\) can eliminate, regardless of where that debt is added.  The \emph{(divisorial) gonality} of a graph \(G\), written \(\gon(G)\), is the minimum degree of a divisor of positive rank on \(G\).  More generally, for \(r\geq 1\) the \(r^{th}\) \emph{divisorial gonality} of a graph \(G\), written \(\gon_r(G)\), is the minimum degree of a divisor of rank at least \(r\) on \(G\).

Given an effective divisor \(D\) of positive rank on a graph \(G\), we know that every divisor of the form \(D-v\) can have debt eliminated; it follows that for every vertex \(v\), \(D\) is equivalent to some effective divisor that places at least one chip on \(v\).  Any divisor with that property will have positive rank, so we can alternatively define \(\gon(G)\) as the minimum degree of a divisor \(D\) such that for any vertex \(v\) we can use chip-firing moves to transform \(D\) into an effective divisor with at least one chip on \(v\).

Several well-studied graph parameters are known to give a lower bound on a graph's gonality, including the tree-width of a graph \cite{debruyn2014treewidth}.  A more powerful lower bound comes from the scramble number of a graph.  A \emph{scramble} \(\mathcal{S}\) on a graph is a collection \(\{V_1,\ldots,V_k\}\) of nonempty vertex sets \(V_i\subset V(G)\) such that \(G[V_i]\) is connected for all \(i\); each \(V_i\) is called an \emph{egg}.  A set \(T\subset V(G)\) is called a \emph{hitting set} for \(\mathcal{S}\) if \(T\cap V_i\neq\emptyset\) for all \(i\); the \emph{hitting number} of \(\mathcal{S}\), written \(h(\mathcal{S})\), is minimum size of a hitting set for \(\mathcal{S}\).  A set \(T\subset E(G)\) is called an \emph{egg-cut} for \(\mathcal{S}\) if \(G-T=(V(G),E(G)-T)\) has at least two connected components completely containing an egg; the \emph{egg-cut number} of \(\mathcal{S}\), denoted \(e(\mathcal{S})\), is the minimum size of an egg-cut for \(\mathcal{S}\).  The \emph{order} of a scramble \(\mathcal{S}\) is defined to be the minimum of these two numbers:
\[\|\mathcal{S}\|=\min\{h(\mathcal{S}),e(\mathcal{S})\}.\]
Finally, the \emph{scramble number} of a graph \(G\), written \(\sn(G)\), the maximum order of any scramble on \(G\).  We have the following powerful result.
\begin{theorem}[Theorem 1.1 in \cite{new_lower_bound}]
    For any graph \(G\), we have \(\sn(G)\leq \gon(G)\).
\end{theorem}

\begin{figure}[hbt]
    \centering
    \includegraphics[width=0.5\linewidth]{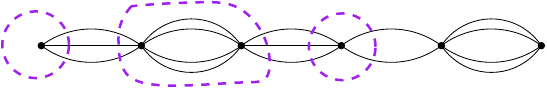}
    \caption{A scramble of order $3$ on a graph.}
\label{figure:scramble_banana_path_example}
\end{figure}
An example of a scramble $\mathcal{S}$ on a graph $G$ is illustrated in Figure \ref{figure:scramble_banana_path_example}.  There are three eggs, two consisting of one vertex each and one consisting of two vertices.  Since the eggs happen to be disjoint, the hitting number is simply the number of eggs: $h(\mathcal{S})=3$.  The egg-cut number is also $3$; the only way to disconnect the graph by deleting fewer edges is to delete the pair of edges to the right of the rightmost egg, but this does not form an egg cut.  Thus the scramble has order $3$, and we know $\sn(G)\geq 3$

A useful graph parameter when studying scramble number is the screewidth of a graph.  A \emph{tree-cut decomposition} of a graph \(G\) is a pair \(\mathcal{T}=(T,\mathcal{X})\), where \(T\) is a tree and \(\mathcal{X}:V(G)\rightarrow V(T)\) is any function of sets.  For clarity, we refer to the vertices and edges of \(T\) as \emph{nodes} and \emph{links}, respectively, reserving the terms \emph{vertices} and \emph{edges} for \(G\).  Given a tree-cut decomposition, we may draw a thickened copy of \(T\) with each vertex of \(G\) drawn inside of the node it maps to; and then edges connecting the vertices of \(G\) can be drawn along the unique path in \(T\) connecting their nodes.  We then assign a weight to each node and link in \(T\) as follows:  the weight of each node is equal to the number of vertices of \(G\) drawn inside of it, plus the number of edges of \(G\) passing through it without either endpoint inside of it.  The weight of each link is equal to the number of edges of \(G\) passing through it.  The \emph{width} of the tree-cut decomposition \(\mathcal{T}\), written \(w(\mathcal{T})\), is equal to the maximum of all these weights.  The \emph{screewidth} of \(G\), written \(\scw(G)\), is defined to be the minimum width of any tree-cut decomposition.  It turns out that screewidth serves as an upper bound on scramble number:
\begin{theorem}[Theorem 1.1 in \cite{screewidth-og}]
    For any graph \(G\), we have \(\sn(G)\leq \scw(G)\).
\end{theorem}
For example, we have our same graph $G$ along with a tree-cut decomposition illustrated in Figure  \ref{figure:tcd_banana_path_example}.  The tree-cut decomposition has three nodes (of weights $1$, $2$, and $3$) and two links (both of weight $3$), meaning the tree-cut decomposition has width $3$, so the graph has screewidth at most $3$.  Combined with our lower bound of $3$ on $\sn(G)$, we know that $3\leq \sn(G)\leq \scw(G)\leq 3$, so $\sn(G)=\scw(G)=3$.  

\begin{figure}[hbt]
    \centering
    \includegraphics{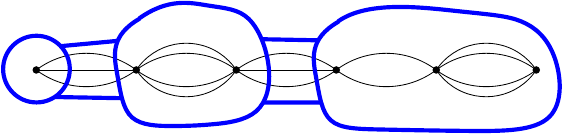}
    \caption{A graph with screewidth at most $3$, as demonstrated by a tree-cut decomposition.}
    \label{figure:tcd_banana_path_example}
\end{figure}

We now turn to the graphs of focus for this paper.  Recall that a banana tree is a multigraph whose underlying simple graph is a tree.  We call the edges connecting any pair of vertices an \emph{edge bunch}.  The banana tree in Figure \ref{figure:banana_tree_example} has $7$ edge bunches, one of size $2$, one of size $4$, and the other five of size $3$.

Given a divisor \(D\) on a banana tree $G$ with adjacent vertices $v$ and $w$, an \emph{adjacency move from $v$ to $w$} is a subset-firing move of the form $S=C_v$, where $C_v$ is the component of the disconnected graph obtained by deleting all edges from $v$ to $w$.  In particular, such an adjacency move sends $|E(v,w)|$ chips from $v$ to $w$.    An adjacency move is called \emph{legal} if it is a legal subset-firing move.

\begin{lemma}
    Let \(D\) and \(D'\) effective divisors on a banana tree \(G\).  Then \(D\sim D'\) if and only if \(D\) can be transformed into \(D'\) by a sequence of legal adjacency moves.
\end{lemma}

\begin{proof}
    The reverse implication is immediate since adjacency moves are composed of chip-firing moves.
    
    Now assume \(D\sim D'\). By Lemma \ref{lemma:level-set}, to show that $D$ can be transformed into $D'$ with a collection of legal adjacency moves, it suffices to show that any one legal subset firing move $S$ can be replicated with a collection of legal adjacency moves.  Consider the set $T$ of all ordered pairs of vertices $(v,w)$ such that $v$ and $w$ are adjacent, $v\in S$, and $v\in S^C$:
    \[T=\{(v,w)\,|\,vw\in E(G), v\in S, v\in S^C\}.\]
    If $D\stackrel{S}{\rightarrow}D''$ is a legal set-firing move from $D$ to $D''$, then 
    \[D''=D-\sum_{(v,w)\in T} |E(v,w)|\cdot v+\sum_{(v,w)\in T}|E(v,w)|\cdot w.\]
This can be rewritten as
\[D''=D+\sum_{(v,w)\in T}(|E(v,w)|\cdot w-|E(v,w)|\cdot v).\]
The change of $|E(v,w)|\cdot w-|E(v,w)|\cdot v$ is achieved by an adjacency move from $v$ to $w$.  Note that this adjacency move must be legal, since $D''\geq 0$ and $v$ does not receive any net chips from the set-firing move.  Thus we may transform $D$ to $D''$ using legal adjacency moves, namely those from $v$ to $w$ ranging over all $(v,w)\in T$.  This completes the proof.
\end{proof}

\begin{remark} In light of this lemma, we could equivalently define chip-firing games on banana trees by only ever allowing chips to move from one vertex $v$ to an adjacent vertex $w$, always in increments of $E(v,w)$.  Such a game could be defined on any graph, but the banana trees are precisely the graphs for which this game would be equivalent.  For if $G$ is a graph whose underlying simple graph has a cycle $C$, say with $v$ and $w$ adjacent $C$ with $a$ edges between them in $G$, then the divisor $a\cdot v$ is not equivalent to the divisor $a\cdot w$.  This can be seen, for instance, using Dhar's burning algorithm \cite{dhar} (see also \cite[Chapter 3]{sandpiles}).
\end{remark}

There is an easy-to-compute upper bound on the gonality of any banana tree in terms of the size of its edge bunches.

\begin{lemma}\label{lemma:easy_upper_bound_gon} If $G$ is a banana tree with $\{a_1,\ldots,a_n\}$ the multiset of the size of its edge bunches, then
\[\gon(G)\leq \lcm(a_1,\ldots,a_n).\]
\end{lemma}
\begin{proof}
    Let $L= \lcm(a_1,\ldots,a_n)$, and consider the divisor $L\cdot v$ for any vertex $v$.  Note that we may move $L$ chips from $v$ to any neighbor $w$; in particular, if $|E(v,w)|=a_i$, we can perform $L/a_i$ legal adjacency moves from $v$ to $w$.  Repeating this argument allows us to move $L$ chips to any vertex in the graph, so $r(L\cdot v)>0$, and thus $\gon(G)\leq L$.
\end{proof}

For instance, the banana tree in Figure \ref{figure:banana_tree_example} has gonality at most $\lcm(2,3,4)=12$ by this result, since we can freely move $12$ chips from any vertex to any adjacent vertex.  Of course, this is not a particularly good upper bound; for instance, we would do better by simply placing a chip on every vertex.  The question of when the bound from Lemma \ref{lemma:easy_upper_bound_gon} is optimal will be explored in Section \ref{section:monotoneandunimodal}.

We will often focus on the case where the underlying tree is a path.  For \(n\geq 0\), the \emph{path graph on \(n+1\) vertices} is the simple graph with vertex set \(\{v_0,v_1,\ldots,v_n\}\), with \(v_i\) and \(v_j\) adjacent if and only if \(|i-j|=1\). A \textit{banana path} is a graph whose underlying simple graph is a path. We denote such banana paths with \(B_A\) where \(A=(a_1,\ldots,a_n)\) is a finite sequence of positive integers that represent the number of edges \(a_i\) in the edge bunch between the vertices \(v_{i-1}\) and \(v_i\). The banana path \(B_{(5, 4, 2, 3, 3, 2)}\) is illustrated in Figure \ref{fig:eg-banana-6}.

\begin{figure}[hbt]
    \centering
    \includegraphics{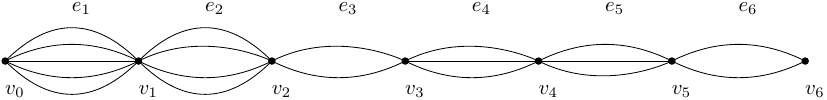}
    \caption{The banana path \(B_{(5, 4, 2, 3, 3, 2)}\), with  $7$ vertices and $6$ edge bunches}
    \label{fig:eg-banana-6}
\end{figure}

We remark that when chip-firing on banana trees, chips can be moved freely from a vertex $v$ to a vertex $w$ if there is exactly one edge $e$ between them.  This is accomplished by firing the set $S$ of vertices in the same component of $v$ in the disconnected graph $G-e$.  Letting $G'$ denote the graph obtained from $G$ by contracting each solitary edge $e$, the graphs $G$ and $G'$ have essentially identical divisorial properties; in particular, $\gon(G)=\gon(G')$  At times it is therefore reasonable to restrict to banana trees where all edge bunches have at least two edges; we refer to such a banana tree as a \emph{ripe} banana tree.

It is worth noting that the assumption of ripeness can lead to different behavior for scramble number or screewidth. For example, \(B_{(100,100)}\) has scramble number at least $3$ (obtained by a scramble where each vertex forms an egg), but  $B_{(100,1,100)}$ has scramble number at most $2$ (witnessed by a tree-cut decomposition with \(v_0\) and \(v_1\) in one bag and \(v_2\) and \(v_3\) in another).

At times it will be useful to study banana paths with a repeating pattern of edge bunch numbers.  Given a sequence \(A=(a'_1,\ldots,a'_s)\) and a nonnegative integer \(n\), we let \(B_{A;n}\) denote the banana path on \(n+1\) vertices with \(a_i=a'_{j}\), where \(j\in\{1,\ldots,s\}\) is congruent to \(i\) modulo \(s\).  In other words, \(B_{A;n}\) has an edge sequence that repeats \(a'_1,\ldots,a'_s\) until it has length \(n\), possibly stopping partway through the sequence \(A\).  We refer to \(B_{A;n}\) as an \emph{\(A\)-monoculture banana path}. The monoculture $B_{(5,3);6}$ is illustrated in  Figure \ref{fig:monoculture-3-5}.

\begin{figure}[hbt]
    \centering
    \includegraphics{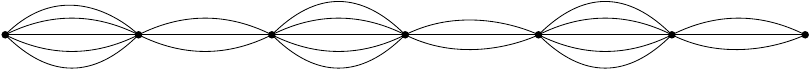} 
    \caption{The monoculture \(B_{A;6}\) where \(A = (5,3)\)}
    \label{fig:monoculture-3-5}
\end{figure}

Not that Lemma \ref{lemma:easy_upper_bound_gon} immediately gives us the following.

\begin{corollary}\label{corollary:upper_bound_monoculture} If $A=(a'_1,\ldots,a'_s)$, then
\[\gon(B_{A;n})\leq \textrm{lcm}(a'_1,\ldots,a'_s).\]
\end{corollary}

\section{Computing gonality and other invariants of banana trees} \label{section:algorithms}

In this section we consider how to compute the gonality of banana trees.  Our first hope might be that a lower bound can be found through the scramble number of graphs.  It turns out scramble number admits a nice combinatorial characterization, which unfortunately lets us quickly see that usually scramble number will be strictly smaller than gonality.

\begin{theorem}\label{theorem:banana_path_scw_sn}
   For a banana tree \(G\), let $k$ be the maximum positive integer such that there is a connected subgraph with at least \(k\) vertices with at least \(k\) edges between each adjacent pair. Then, \(\mathrm{sn}(G)=\mathrm{scw}(G)=k\).
\end{theorem}

\begin{proof} First  we construct a scramble of order \(k\). Choose $k$ vertices $v_{i_1},\ldots,v_{i_k}$ forming a connected subgraph $G'$, where between each adjacent pair of vertices there are at least $k$ edges.  Let $\mathcal{S}$ be the scramble $\{V_1,\ldots,V_k\}$ where $V_{j}=\{v_{i_j}\}$; that is, the scramble whose eggs are precisely the vertices of $G'$.  We have $h(\mathcal{S})= k$
 since there are $k$ disjoint eggs, and the smallest egg-cut is found by choosing the smallest edge bunch in $G'$ and deleting all edges in it; thus $e(\mathcal{S})\geq k$.  Thus we have $\sn(G)\geq \|\mathcal{S}\|\geq k$.


We now construct a tree-cut decomposition of width at most \(k\).  Start with any leaf vertex, \(w_{i_1}\), and place it in a bag.  Let $w_{i_2}$ be its neighbor. Add \(w_{i_2}\) to this bag if \(E(w_{i_1},w_{i_2})>k\); otherwise, start a new bag for \(w_{i_2}\).  Repeat this process for each neighbor of $w_{i_2}$:  add it to $w_{i_2}$'s bag if there are more than $k$ edges between them, or give it its own new bag.  Repeat this process until all vertices have been placed in a bag.    Arrange these bags in a tree so that two bags are adjacent if and only if they have an adjacent vertex, giving us a tree-cut decomposition \(\mathcal{T}\).  By our choice of when we create a new bag, the maximum link adhesion is at most \(k\). There is no bag adhesion, and the maximum size of a bag is at most \(k\):  a larger bag could only result from a connected subgraph of at \(k+1\) consecutive vertices with at least \(k+1\) edges connecting each adjacent pair, contradicting our choice of \(k\). As a result, we have \(\mathrm{scw}(G) \leq w(\mathcal{T})\leq  k\).

 Taken together, we have
 \[k\leq \sn(G)\leq \scw(G)\leq k,\]
 implying that \(\mathrm{sn}(G)=\mathrm{scw}(G)=k\).
\end{proof}

It is very quick to see from this argument that gonality is typically strictly larger than scramble number.  Assume $\sn(G)=k$, and pick an induced subgraph $G'$ on $k$ vertices with at least $k$ edges between each adjacent pair.  If any pair has strictly more than $k$ edges, then we cannot freely move $k$ chips around on this subgraph, and the only hope to have gonality $k$ would be to place $1$ chip on each vertex in $G'$.  However, if any edge bunch in $G$ outside of $G'$ has more than one edge, then this divisor does not have positive rank.  Even if all edge bunches in $G'$ consist of precisely $k$ edges, there is no guarantee that $k$ chips placed on a single vertex  of $G'$ would be able to reach every vertex of $G$.

We now develop an tool that will help in computing the gonality of certain banana trees.  
Let $n\geq 2$, let $v$ be a vertex of a banana tree, and let $k$ be a nonnegative integer.  Define the function \(f(G,v,k)\) as the minimum degree of an effective divisor $D$ such that $r(D+k\cdot v)>0$.  Informally, 
\(f(G,v,k)\) is the minimum number of chips we would need to place on \(G\) to create a positive rank divisor, assuming we get to start with \(k\) chips on $v$.  In particular, \(f(G,v,0)=\mathrm{gon}(G)\) for any vertex $v$.  In the case that $G$ is a banana path, we let $f(G,k)$ denote $f(G,v_0,k)$; that is, if no vertex is specified, we assume we are considering the leftmost vertex.

\begin{proposition}\label{prop:recursion_step_one}  Let $v$ be a leaf vertex of a banana tree $G$, and let $G'=G-v$ be the banana tree obtained by deleting $v$.  Let $w$ be the unique neighbor of $v$ in $G$, and let $a=|E(v,w)|$. Then
    \[f(G,v,0)=\min \{1 + f(G', w, 0), a+f(G', w,a)\}\]
    and
    \[f(G,v,k)=\min\{f(G',w,a\cdot \lfloor k/a\rfloor),a\cdot\lceil k/a\rceil-k+f(G',w,a\cdot \lceil k/a\rceil)\}\]
for \(k>0\).
\end{proposition}

Before we present a formal proof, we provide some intuition behind these formulas, starting with $f(G,v,0)=\gon(G)$.  If we have a positive rank divisor, then either we will use the edge bunch incident to $v$ to move chips around, or we will not.  If we do not use this edge bunch, we must have one chip on $v$ (at least one for positive rank, at most one for minimality), and then a winning placement on the remainder of the graph, which operates independently.  If we do use this edge bunch, then at some point the vertex $v$ must have at least $a$ chips on it.  By minimality, we can assume it has a multiple of $a$ chips on it (otherwise some chips would always be stranded there), and by firing $v$ some number of times we may in fact assume are exactly $a$ chips on $v$ at some point.  We can fire these along to $v$, and then optimize on the remainder of the graph.

For $f(G,v,k)$ with $k>0$, there are essentially two strategies:  if we do not add any chips to $v$ to obtain a positive rank divisors, then the best we can do is move as many chips of the $k$ chips from $v$ to $w$, and then focus on the remainder of the graph.  If we do add some chips to $v$, by minimality we should aim to get a multiple of $a$ chips on $v$ (otherwise some will be stranded, and we may as well have added fewer), fire them along to $w$, and then optimize from there.

\begin{proof}[Proof of Proposition \ref{prop:recursion_step_one}]
    First we prove the claim for $f(G,v,0)$.  Let $D'$ be a positive rank divisor on $G'$ with $\deg(D')=\gon(G')=f(G',w,0)$.  Then, the divisor $D=1\cdot v+D'$ on $G$ has positive rank: it already places a chip on $v$, and to place a chip on another vertex $u$ we can follow the same legal adjacency moves that would allow $D'$ to place a chip on $u$ in  $G'$.  Thus $f(G,v,0)=\gon(G)\leq \deg(D)=1+\gon(G')=1+f(G',w,0)$.  Another positive rank divisor $D$ on $G$ can be constructed by choosing an effective divisor $D'$ on $G'$ with $r(D')>0$, $D'(w)\geq a$, and $\deg(D')=a+f(G',w,a)$, and defining $D=D'+0\cdot v_0$ on $G$.  This has positive rank since the legal adjacency move from $w$ to $v$ places $a$ chips on $v_0$, and a chip can be placed on any other vertex by mirroring adjacency moves for $D'$ on $G'$ since $r(D')>0$.  Thus $f(G,v,0)=\gon(G)\leq \deg(D)=\deg(D')=a_1+f(G',w,a_1)$.  This gives us
    \[f(G,v,0)\leq \min \{1 + f(G',w, 0), a+f(G',w, a)\}.\]
    
    To lower bound $f(G,0)$, let $D$ be a $v$-reduced divisor on $G$ with $r(D)>0$ and $\deg(D)=\gon(G)=f(G,v,0)$.  Consider $D_{w}$, the unique $w$-reduced divisor equivalent to $D$.  If $D=D_{w}$, then $D_{w}(v)=D(v)\geq 1$.  Moreoever, the divisor $D':=D-D(v)\cdot v$ has positive rank when considered as a divisor on $G'$.  It follows that $\deg(D)=D(v)+\deg(D')\geq 1+\gon(G')=1+f(G',w,0)$.
    If $D\neq D_{w}$, then since $D$ is $v$-reduced we obtain $D_{w}$ from $D$ by chip-firing $v$ a positive number of times.  It follows that $D_{w}(w)\geq a$.  We also have that $D'=D_{w}-D_{w}(v)\cdot v$ has positive rank when considered on $G'$; since it has at least $a$ chips on $w$, we know $\deg(D')\geq a+f(G',w,a)$.  Thus $\deg(D)=D_{w}(v)+\deg(D')\geq a+f(G',w,a) $.  Between these two cases, we know that $f(G,v,0)=\deg(D)$ is lower bounded by at least one of $1+f(G',w,0)$ and $a_1+f(G',w,a_1)$.  Since it is also upper bounded by their minimum, it must equal their minimum, giving us the claimed formula.

   Now we handle the case where $k>0$.  First we note that for any divisor with $k$ chips on $v$, we may move $a\cdot \lfloor k/a\rfloor$ chips to $w$ by firing $v$ a total of $\lfloor k/a\rfloor$ times.  Let $D'$ on $G'$ have $D'(w)\geq a\cdot \lfloor k/a\rfloor$ with $r(D')>0$ and $\deg(D')=a\cdot \lfloor k/a\rfloor+f(G',w,a\cdot \lfloor k/a\rfloor)$.  Then, $k\cdot v+(D'-a\cdot \lfloor k/a\rfloor\cdot w)$ has positive rank on $G$: we may freely move $\lfloor k/a\rfloor$ chips to $w$, and mirror chip movements from $D'$ on $G'$ from there.  Thus $f(G,v,k)\leq \deg(D')-a\cdot \lfloor k/a\rfloor=f(G',w,a\cdot \lfloor k/a\rfloor)$.  For the other upper bound, let $D'$ on $G'$ be an effective divisor such that $r(D'+a\cdot \lceil k/a\rceil \cdot w)>0$ and $\deg (D')=f(G',w,a\cdot \lceil k/a\rceil)$.  Consider the divisor $D=a\cdot\lceil k/a\rceil\cdot v+D'$ on $G$. We may freely move $a\cdot\lceil k/a\rceil$ chips from $v$ to $w$, and from there we may emulate the firing moves of $D'$ on $G'$ to place chips elsewhere, so $r(D)>0$.  Note also that $D$ places at least $k$ chips on $v$.  Thus, $f(G,v,k)\leq \deg(D)-k=a\cdot\lceil k/a\rceil-k+\deg(D')=a\cdot\lceil k/a\rceil-k+f(G',w,a\cdot \lceil k/a\rceil).$  We therefore have
      \[f(G,v,k)\leq \min\{f(G',w,a\cdot \lfloor k/a\rfloor),a\cdot\lceil k/a\rceil-k+f(G',w,a\cdot \lceil k/a\rceil)\}.\]

     For the lower bound, let ${D}$ be an effective divisor on $G$ such that $r(k\cdot v+{D})>0$ and $\deg({D})=f(G,v,k)$.  Let $\overline{D}=k\cdot v+{D}$.
     Consider $\overline{D}_{w}$, the unique $w$-reduced divisor equivalent to $\overline{D}$.  First assume $D(v)=0$ (so $\overline{D}(v)=k$).  We know that $\overline{D}_{w}(w)\geq a\cdot \lfloor k/a\rfloor$, since $v$ could be fired $\lfloor k/a\rfloor$ times.  Note that $\overline{D}_{w}-\overline{D}_{w}(v)\cdot v$ has positive rank considered as a divisor on $G'$.  Note also that $k=\overline{D}_{w}(v)+\lfloor k/a\rfloor\cdot a$. Thus 
     \begin{align*}
         k+\deg(D)=&\deg(\overline{D})
\\\geq&\deg(\overline{D}_{w})\\=&\overline{D}_w(v)+a\cdot \lfloor k/a\rfloor+f(G',w,a\cdot \lfloor k/a\rfloor) \\=&k+f(G',w,a\cdot \lfloor k/a\rfloor).
     \end{align*}  It follows that $\deg(D)\geq f(G',w,a\cdot \lfloor k/a\rfloor)$.
     
Now assume $D(v)>0$. We claim that $D(v)+k$ is a multiple of $a$.  Otherwise, we could have added one fewer chip to $v$, and still been able to move just as many chips to $w$, at which point the behavior of the divisor would have been the same. The smallest multiple of $a$ that is at least $k>0$ is $\lceil k/a\rceil$. Thus $\overline{D}(v)=D(v)+k\geq a\cdot \lceil k/a\rceil$.  It follows that $\overline{D}_{w}(w)\geq a\cdot \lceil k/a\rceil$.  We also have that $\overline{D}_{w}(v)=0$, since $D(v)+k$ is a multiple of $a$.  Letting $\overline{D}'$ be the divisor $\overline{D}_{w}$ considered on $G'$, we have that $r(\overline{D}')>0$.  Since $\overline{D}'(w)\geq a\cdot \lceil k/a\rceil$, it follows that $\deg(\overline{D})=\deg{(\overline{D'})}\geq a\cdot \lceil k/a\rceil+f(G',w,a\cdot \lceil k/a\rceil)$.  Now, $\deg(D)+k\geq a\cdot \lceil k/a\rceil+f(G',w,a\cdot \lceil k/a\rceil)$, so $\deg(D)\geq a\cdot \lceil k/a\rceil-k +f(G',w,a\cdot \lceil k/a\rceil)$.

Taken together, we have that $f(G,k)=\deg(D)$ is lower bounded by at least one of $f(G',a\cdot \lfloor k/a\rfloor)$ and $a\cdot \lceil k/a\rceil-k +f(G',w,a\cdot \lceil k/a\rceil)$. Since it is also upper bounded by their minimum, it must be equal to their minimum, completing the proof.
\end{proof}

Since this theorem requires the vertex $v$ to be a leaf, it is not possible to apply it recursively to all banana trees, since the neighbor $w$ might not be a leaf.  However, in the case of banana paths, the vertex $w$ will always be a leaf, meaning we may recursively compute the gonality of any banana path with this result.  This allows us to prove Theorem \ref{theorem:polytime_gonality}, that the gonality of a banana path on $n+1$ vertices can be computed in $O(n^3)$ time.
\begin{proof}[Proof of Theorem \ref{theorem:polytime_gonality}]  
    The algorithm implements the recursion shown in Proposition \ref{prop:recursion_step_one} with memoization. Furthermore, because the gonality of any $n$-vertex graph $G$ is at most $n$, any edge bunch of size $>n$ effectively splits the instance into two (it is never efficient to send chips along such a bunch). Therefore, $f(G, k) = 0$ for all $k \geq n^2 + 1$ since at most $n$ chips are left behind as remainder per fire. 

    \begin{algorithm}[H]
    \caption{} 
    \label{banana-path-gonality}
        \begin{algorithmic}
            \State \textbf{\underline{Banana Path Gonality}}($G$):
            \Indent
                \State subgraphs $\gets [G]$, gonality $\gets 0$
                \While{$\exists G_i \in$ subgraphs which contains edge bunch of size $>|V(G_i)|$}
                    \State Split $G_i$ into $G^1_i$ and $G^2_i$ along heavy edge bunch
                \EndWhile{}
                \State
                \For{$G_i \in $ subgraphs}
                    \State $n_i \gets |V(G_i)|$, $f \gets \mathbb{Z}^{n_i \times (n_i^2 +1)}$
                    \State gonality $+=$ Recursion($G_i$, $0$, $f$)
                \EndFor{}
                \Return gonality 
            \EndIndent
            
            \State
            \State \textbf{\underline{Recursion}}($G, k, f$):
            \Indent
                \If{$f[n][k]$ is not empty} 
                    \Return $f[n][k]$ \Comment{Memoization}
                \EndIf{}
                \State
                \If{$n = 1$} \Comment{Base Case}
                    \If{$k= 0$}
                        $f[n][k] = 1$
                    \Else{}
                        $f[n][k] = 0$
                    \EndIf{}
                    \State \Return $f[n][k]$
                \EndIf{}
                \State
    
                \If{$k=0$} 
                    $f[n][k] = \min\{1 + \text{Recursion}(G', 0, f), a + \text{Recursion}(G', a, f) \}$ \Comment{Recurse}
                \Else{}
                    \State $f[n][k] = \min\{\text{Recursion}(G',a\cdot \lfloor k/a\rfloor,f),a\cdot\lceil k/a\rceil-k+\text{Recursion}(G',a\cdot \lceil k/a\rceil, f)\}$
                \EndIf{}
            \EndIndent
            
        \end{algorithmic} 
    \end{algorithm} 

    \noindent The time complexity of the algorithm is given by the number of unique subproblems solved in the recursion scaled by the time required to solve each one. There are 
    \[
    \sum_{\text{subgraphs} \{G_i\}_{i=1}^{m}} n_i \cdot(n_i^2 + 1) \leq O(n^3)
    \]
    many subproblems (where we have used the fact that $\sum_{i} n_i = n$). Each subproblem takes constant time to solve, therefore the total time complexity of Algorithm \ref{banana-path-gonality} is $O(n^3)$. 
\end{proof}

 We close this section with a treatment of \emph{banana stars}.  These are graphs whose underlying simple graph is the star tree, consisting of one vertex $v_0$ adjacent to all others (none of which are adjacent to one other).  We represent a banana star as $S_{\left(a_1^{r_1},a_2^{r_2},\ldots,a_k^{r_k}\right)}$, where $a_1>a_2>\cdots>a_k$ and $r_i$ denotes the number of edge bunches of size $a_i$.  The banana star $S_{(4^1,3^3,2^3,1^2)}$ is illustrated in Figure \ref{figure:banana_star_example}.

 \begin{figure}[hbt]
     \centering
     \includegraphics{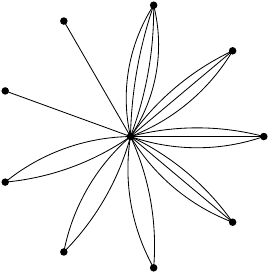}
     \caption{The banana star $S_{(4^1,3^3,2^3,1^2)}$.}
     \label{figure:banana_star_example}
 \end{figure}

\begin{theorem}\label{theorem:banana_stars}  Let $1\leq a_1\leq \cdots \leq a_n$, and let $G=S_{\left(a_1^{r_1},a_2^{r_2},\ldots,a_k^{r_k}\right)}$.  We have
\[\gon(G)=\min_{0\leq j\leq k}\left\{\sum_{i=1}^j r_i+a_{j+1}\right\},\]
where we take $a_{k+1}$ to be $1$.
\end{theorem}

\begin{proof}
    First we show that for any $j$ there is a positive rank divisor of degree $\sum_{i=1}^j r_i+a_{j+1}$.  Place one chip on each vertex $w$ with $|E(v_0,w)|>a_{j+1}$; there are $\sum_{i=1}^j r_i$ such vertices.  Then place $a_{j+1}$ chips on $v_0$.  Every vertex $w$ either already has a chip, or is adjacent to $v_0$ with at most $a_{j+1}$ edges connecting them, so we can move a chip to $w$ with a legal adjacnecy move from $v_0$.  This divisor has positive rank, so $\gon(G)\leq \sum_{i=1}^j r_i+a_{j+1}$ for all $1\leq j\leq k$.

    For the lower bound, suppose $D$ has positive rank with $\deg(D)=\gon(G)$.  Transform $D$ into the effective divisor $D'$ that maximizes the number of chips on $v_0$. If $D'$ places a chip on every other vertex $w$, then $\deg(D)=n+1=\sum_{i=1}^k r_i+a_{k+1}\geq \min_j\left\{\sum_{i=1}^j r_i+a_{j+1}\right\}$, and our claim holds.  Thus assume $D'$ fails to place a chip on at least one vertex; pick $w$ to be such a vertex with maximum possible  $|E(v_0,w)|$.  Note that   
    $D'(v_0)$ must be at least $|E(v_0,w)|$, since we must be able to move chips from $v_0$ to $w$ through a legal adjacency move.  Choose $j$ such that  $a_{j+1}=|E(v_0,w)|$.  Then every vertex with $a_1,a_2,\ldots,a_j$ edges connecting it to $v_0$ must have a chip, and thre are $\sum_{i=1}^j r_i$ such vertices. It follows that
    \[\deg(D)\geq \sum_{i=1}^j r_i+a_{j+1}\geq \min_j\left\{\sum_{i=1}^j r_i+a_{j+1}\right\},\]
    completing our lower bound.
\end{proof}

As one might expect, these graphs typically have scramble number distinct from gonality.

\begin{proposition}\label{proposition:banana_stars_scramble}  Let $1\leq a_1\leq \cdots \leq a_k$, and let $G=S_{\left(a_1^{r_1},a_2^{r_2},\ldots,a_k^{r_k}\right)}$.  Then 
\[\sn(G)=\max_{1\leq \ell\leq k}\min\left\{a_\ell,1+\sum_{i=1}^\ell r_i\right\}\]
\end{proposition}

\begin{proof}

For a lower bound, we note the scramble $\mathcal{S}$ whose eggs are the singletons $\{v\},\{w_1\},\ldots,\{w_m\}$, where $w_1,\ldots,w_m$ are all leaf vertices with $|E(v,w_i)|\geq a_\ell$, will have order $\min\{a_\ell,1+m\}=\min\{a_\ell,1+\sum_{i=1}^\ell r_i\}$.  Thus
\[\sn(G)\geq \min\left\{a_\ell,1+\sum_{i=1}^\ell r_i\right\}\]
for every choice of $\ell$.

For an upper bound, we refer to Therem \ref{theorem:banana_path_scw_sn}, which states that the scramble number of $G$ is the maximum $k$ such that $G$ has a connected subgraph $G'$ on at least $k$ vertices with at least $k$ edges between any pair of vertices.  If every edge bunch has size $1$, we have $\sn(G)=1$ and our claim holds\footnote{This is because the scramble number of any tree is $1$; see \cite[Corollary 4.2]{new_lower_bound}.}.  Otherwise the vertex set of our connected subgraph $G'$ will consist of $v_0$ together with some leaves. Let $a_\ell$ be the smallest edge bunch appearing in $G'$.  Without loss of generality we may assume that $G'$ contains all leaves $w$ with $|E(v_0,w)|\geq a_\ell$, as adding such leaves will neither decrease the number of vertices nor decrease the size of the smallest edge bunch.  The scramble achieving scramble number has as its eggs the individual vertices of $G'$, and the order of the scramble is $\min\{a_\ell, 1+\sum_{i=1}^\ell r_i\}$.  Thus $\sn(G)$ is equal to  $\min\{a_\ell, 1+\sum_{i=1}^\ell r_i\}$ for some choice of $\ell$. Since we have a lower bound of $\sn(G)$ of the same formula for \emph{all} choices of $\ell$, our claim follows.
\end{proof}

\section{Monotone and unimodal banana graphs} \label{section:monotoneandunimodal}

In this section we find the gonality of sufficiently long banana graphs with repeating edge patterns, and prove that gonality can increase by an arbitrary amount upon the deletion of a single edge.  Both results rely upon the following lemma.

\begin{lemma}\label{lemma:lcm_string}
Let $G$ be a banana path with a string of adjacent vertices $v_k,\ldots,v_{k+\ell}$ with a repeating pattern of edge bunch sizes $a_1,\ldots,a_s$ among those $\ell+1$ vertices, and set $L=\lcm(a_1,\ldots,a_s)$. If $\ell+1\geq sL^2$, and $D$ is a positive rank divisor on $G$, then some effective divisor equivalent to $D$ places at least $L$ chips on the vertices $v_k,\ldots,v_{k+\ell}$.
\end{lemma}

\begin{proof}
    Without loss of generality, we may assume that $D$ is $v_k$-reduced.  We claim that  $D$ places $L$ chips on $v_k,\ldots,v_{k+\ell}$. Suppose for the sake of contradiction there are fewer chips. Since $\ell+1\geq sL^2$, there must be $sL$ adjacent vertices $v_r,\ldots,v_{r+sL-1}$ with zero chips on them, where $k\leq r<r+sL-1\leq k+\ell$.

    Transform $D$ into $D'$, where $D'$ is $v_{r}$-reduced.  Since $D$ is $v_k$-reduced, all legal adjacency moves in this process are are left-to-right, and take place entirely within $v_k,\ldots,v_{r+sL-1}$.  Once we reach as many chips as possible on $v_r$, we know there are at most $L-1$ chips on $v_r$, since no new chips have been introduced to $v_k,\ldots,v_{k+\ell}$.  Now we transform $D'$ into $D''$, where $D''$ is $v_{r+sL-1}$-reduced.  Similar to before, this can be accomplished through left-to-right legal moves, all within $v_r,\ldots,v_{r+sL-1}$. 
    
    However, as chips are moved across the next $s$ vertices, at least one chip must be left behind, since the number of chips does not divide all of $a_1,\ldots,a_s$.  This occurs repeatedly over the $L$ sets of $s$ vertices making up the consecutive vertices with no chips, meaning that we will have lost all the at most $L-1$ chips by the time we attempt to move chips from $v_{r+sL-2}$ to $v_{r+sL-1}$.  This means that $D''$ is $v_{r+sL-1}$-reduced but has no chips on $v_{r+sL-1}$, a contradiction to $r(D'')>0$.  Thus $D$ places at least $L$ chips on the vertices $v_k,\ldots,v_{k+\ell}$, as desired.
\end{proof}

\begin{corollary}\label{corollary:lcm_equality}  Let $A=(a_1,\ldots,a_s)$, let $L=\textrm{lcm}(a_1,\ldots,a_s)$, and let $n\geq sL^2$.  Then the gonality of the $A$-monoculture banana path on $n+1$ vertices has gonality $L$:
\[\gon(B_{A;n})=\lcm(a_1,\ldots,a_s).\]
\end{corollary}

\begin{proof}
    By Corollary \ref{corollary:upper_bound_monoculture}, \(\gon(B_{A;n})\leq L\).  For the lower bound, we apply Lemma \ref{lemma:lcm_string} using all vertices $v_0,\ldots,v_{n}$, implying that any positive rank divisor has an equivalent effective divisor placing $L$ chips on those vertices.  Thus $\gon(B_{A;n})\geq L$.
\end{proof}

We now present an example of gonality increasing under taking a subgraph.

\begin{example}\label{example:increase_of_one}
    Consider $B_{(3,3,3)}$.  This graph has gonality $3$, as can be proved with the upper bound from Corollary \ref{corollary:upper_bound_monoculture} and a lower bound from scramble number.  However, if we delete an edge from the middle bunch to obtain $B_{(3,2,3)}$, we claim the gonality becomes $4$.  Indeed, suppose there exists $D$ on $B_{(3,2,3)}$ with $r(D)>0$, $\deg(D)=3$, and $D\geq 0$.  Without loss of generality we can take $D$ to be $v_0$-reduced, so $D(v_0)\geq 1$.  Note that $D(v_0)$ cannot be $3$, because then $D=3\cdot v_0$, but $r(3\cdot v_0)=0$ (for instance, it cannot place a chip on $v_3$).  By minimality, then $D(v_0)=1$.  A similar argument shows that $D_{v_3}(v_3)=1$; but since a solitary chip cannot be moved to $v_3$ through legal adjacency moves, we have $D(v_3)=1$.  Thus $D$ is $v_0+v_i+v_3$ for either $i=1$ or $i=2$; but this divisor does not have positive rank, since it cannot place a chip on $v_j$ with $j\neq i$, $j\in \{1,2\}$.  Having reached a contradiction, we know that $B_{(3,2,3)}$ has gonality at least $4$.  Then note that $r(3v_0+v_3)>0$.
    
    Another way to compute the gonality is to apply Proposition \ref{prop:recursion_step_one} iteratively to $B_{(3,2,3)}$.  The first iteration gives
\begin{align*}
    \gon(B_{(3,2,3)})=&f(B_{(3,2,3)},0)
    \\=&\min\{1+f(B_{(2,3)},0),3+f(B_{(2,3)},3).\}
\end{align*}
Applying the result to $f(B_{2,3},0)$ gives
\[f(B_{(2,3)},0)=\min\{1+f(B_{(3)},0),2+f(B_{(3)},2)\}=\min\{1+2,2+1\}=3,\]
and to $f(B_{(2,3)},3)$ gives
\begin{align*}f(B_{(2,3)},3)=&\min\{f(B_{(3)},2\cdot \lfloor 3/2\rfloor),2-(3\mod 2)+f(B_{(3)},2\cdot\lceil 3/2\rceil)\}
\\=&\min\{f(B_{(3)},2),1+f(B_{(3)},4)\}=\min\{1,1+0\}=1.
\end{align*}
Plugging back in we find
\[\gon(B_{(3,2,3)}=\min\{1+f(B_{(2,3)},0),3+f(B_{(2,3)},3)=\min\{1+3,3+1\}=4.\]
Thus we have an example of a graph $G$ and an edge $e$ such that $\gon(G-e)=\gon(G)+1$.
\end{example}

We now prove Theorem \ref{theorem:peeling_edges}.  The broad idea is to start with a long $(a,b)$-monoculture banana path on $n+1$ vertices with gonality $\lcm(a,b)$, and to delete an edge near the middle.  By carefully choosing $a$, $b$, and $n+1$, we can make it so that each side of the graph needs $\lcm(a,b)$ chips to achieve a positive rank divisor, \emph{and} so that $r$ additional chips are needed near the middle to make up for chips that cannot be transported across the middle edge bunch.

\begin{proof}[Proof of Theorem \ref{theorem:peeling_edges}]     Let \(r\) denote the increase in gonality we want the graph to have by deleting a specific edge. If $r=1$, Example \ref{example:increase_of_one} provides our desired example.  Now let $r\geq 2$.  We pick an integer \(a\) such that
\[        r \leq \left\lfloor\frac{a-1}{2}\right\rfloor\]
with the additional requirements that \(r \nmid a-1\); this is always possible since $r\geq 2$. Using Dirichlet's theorem on primes in arithmetic progressions \cite{dirichlet}, we know that there exist infinitely many primes \(b\) for which \(r \equiv b\, \mathrm{mod}\, (a-1)\). We  pick the first prime $b$ of this infinite sequence such that \(b \nmid a\) and \(b \nmid a-1\).  We may write $b=r+k(a-1)$.

Let $n=4(ab)^2+1$, and consider the $(a,b)$-monoculture banana path \(B_{(a,b);n}\), which has $n+1$ vertices alternating between $a$ and $b$ edges in each bundle. By Corollary \ref{corollary:lcm_equality}, we know that the gonality of this graph is \(\textrm{lcm}(a,b)=ab\).

Now let $G$ be the graph obtained by deleting an edge from the central edge bundle with $a$ edges, decreasing the number from $a$ to $a-1$.  We claim that $\gon(G)=ab+r$.  For a positive rank divisor $D$ of that degree, set $D=ab\cdot v_{2(ab)^2}+r\cdot v_{2(ab)^2+1}$, which places $ab$ chips on the left vertex $v_{2(ab)^2}$ of the central edge bunch, and $r$ chips on the right vertex $v_{2(ab)^2+1}$. Legal adjacency moves let us move $ab$ chips to any vertex to the left of the central bunch.  For vertices to the right, starting from $D$ first move as many chips from the left vertex of the central edge bunch to the right vertex.  Note that $ab\equiv 1\cdot r\equiv r\mod (a-1)$, so $ab$ is $r$ more than a multiple of $a-1$.  We may therefore move $ab-r$ chips to the right vertex, yielding the equivalent divisor $r\cdot v_{2(ab)^2}+(ab-r+r)\cdot v_{2(ab)^2+1}=r\cdot v_{2(ab)^2}+ab\cdot v_{2(ab)^2+1}=$. Since there are $ab$ chips on $v_{2(ab)^2+1}$, we may move chips to cover any vertex on the right side.  Thus $\gon(G)\leq\deg(D)=ab+r$.

Now we show $\gon(G)\geq ab+r$.  Suppose for the sake of contradiction that there exists a positive rank divisor $D$ of degree $ab+r-1$ on $G$.  Since $v_0,\ldots,v_{2(ab)^2}$ satisfy the hypotheses of Lemma \ref{lemma:lcm_string} with $s=2$, We know that $D\sim D'$ for some effective divisor $D'$ that places at least $ab$ chips to the left of the deleted edge.  Similarly, $D\sim D''$ for some effective divisor $D''$ that places at least $ab$ chips to the right  of the deleted edge.  Consider a sequence of legal adjacency moves transforming $D'$ into $D''$.  Among these moves must be adjacency moves moving at least $ab-r+1$ chips across the central edge bunch, 
 since $D'$ places at most $r-1$ to the right of this edge bunch.  Note that $ab-r+1\equiv 1\cdot r-r+1\equiv 1\mod a-1$, and since we must move in increments of $a-1$ we must in fact move at least $ab-r+1+(a-2)=ab+(a-1)-r$ chips through these legal adjacency moves.  Since $r\leq \lfloor (a-1)/2\rfloor$, we know $(a-1)-r>r-1$, so this is strictly more than $ab+r-1$ chips, a contradiction to our divisor having degree $ab+r-1$.  Thus $\gon(G)\geq ab+r$.

 In summary, we have that $\gon(G)=ab+r$  and $\gon(B_{(a,b);n})=ab$, giving us our claimed example.
\end{proof}

\begin{example}
    For a concrete example, suppose we wish to find a drop of $r=2$ when deleting an edge.  We pick an integer $a$ such that $2\leq \left\lfloor\frac{a-1}{2}\right\rfloor$ and $2\nmid a-1$; the smallest possible choice of $a$ is $a=6$.  We then choose $b$ to be the smallest prime such that $b\equiv 2\,\mod\,5$ (making sure that $b\nmid 6$ and $b\nmid 5$), which is $b=7$.  Finally we choose $n=4(ab)^2+1=4(6\cdot 7)^2+1=7057$, so our graph will have $n+1=7058$ vertices.  Thus our banana path is $B_{(6,7);7057}$, and $G$ is the graph obtained by deleting an edge from the middle edge bunch (between vertices $v_{3528}$ and $v_{3529}$), reducing the number of edges in that edge bunch from $6$ to $5$.

\begin{figure}[hbt]
    \centering
    \includegraphics{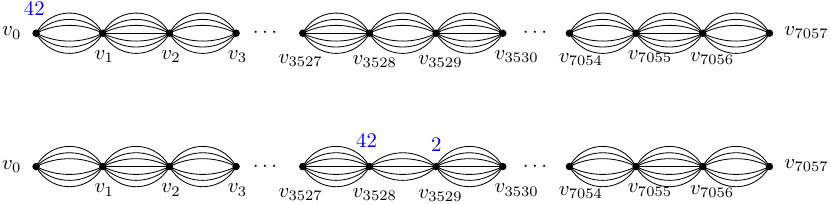}
    \caption{The banana path $B_{(6,7);7057}$, with gonality $42$; and the banana path $G$ obtained from deleting one edge from the middle of the graph, with gonality $44$.}
    \label{figure:example_drop_2}
\end{figure}

    As argued in the previous proof, $\gon(B_{(6,7);7057})=6\cdot 7=42$, and $\gon(G)=6\cdot 7+2=44$.  For the first graph, $42\cdot v_0$ achieves gonality.  A divisor on $G$ achieving gonality is $42v_{3528}+2v_{3529}$.  Note that $42v_{3528}+2v_{3529}\sim 2v_{3528}+42v_{3529}$ since we can perform $8$ legal adjacency moves from $v_{3528}$ to $v_{3529}$. Once there are $42$ chips on the vertex to the left or the right of the central edge bunch, any vertex can receive one (in fact, $42$) chips through legal adjacency moves, since $42$ is a multiple of both $6$ and $7$.  These graphs, and divisors of positive rank and minimum possible degree, are illustrated in Figure \ref{figure:example_drop_2}.
\end{example}

\section{Banana trees and Brill-Noether theory}\label{section:brill}

In  this section we prove that every banana tree satisfies the gonality conjecture, and moreover that only finitely many ripe banana trees achieve the conjectured upper bound of $\left\lfloor \frac{g+3}{2}\right\rfloor$.

\begin{proof}[Proof of Theorem \ref{theorem:brill_noether_bound}]
Since contracting edge bunches of size one affects neither genus nor gonality nor the size of the largest edge bunch (except in the case that $G$ is a tree), we may assume without loss of generality that $G$ is a ripe banana tree.

We prove the claim by strong induction on \(n\), where \(n+1\) is the number of vertices. For \(n=0\), the banana tree has one vertex and gonality equal to \(1\); since this graph has \(g=0\), this matches the upper bound of \(\left\lfloor \frac{g+3}{2}\right\rfloor=\left\lfloor \frac{0+3}{2}\right\rfloor=1\).  
    For \(n=1\), the banana tree has two vertices joined by \(a\) edges, and has genus equal to \(g=a-2+1=a-1\). By our ripeness, assumption, \(a\geq 2\), and so the graph has gonality \(2\), which is at most \(\left\lfloor \frac{g+3}{2}\right\rfloor=\left\lfloor \frac{a+2}{2}\right\rfloor\), with equality precisely when \(a=2\) or \(a=3\).  Thus in our base cases, our upper bound holds, and a necessary condition for equality is that no edge bunch is larger than size $4$.

Now let $n\geq 2$, and assume our claim holds for all banana trees on up to \(n+1-1=n\) vertices.  Let \(G\) be a banana tree on \(n+1\) vertices.  Choose a leaf $v$ of $G$ with unique neighbor $w$, and let $a=|E(v,w)|$.  Let $G'=G-v$.  By the inductive hypothesis, $\gon(G')\leq  \lfloor (g(G')+3)/2\rfloor=(g(G)-(a-1)+3)/2\rfloor=\lfloor(g(G)-a+4)/2\rfloor$. 
Note that $\gon(G)\leq 1+\gon(G')$. If $a\geq 3$, then we have
\[\gon(G)\leq 1+\gon(G')=1+\lfloor(g(G)-a+4)/2\rfloor=\lfloor(g(G)-a+6)/2\rfloor\leq\lfloor(g(G)-3)/2\rfloor. \]
Note that a necessary condition for equality is to have $a\leq 4$; for larger $a$, at least one of the above inequalities is strict.  Similarly, a necessary condition is $\gon(G')=  \lfloor (g(G')+3)/2\rfloor$; by the inductive hypothesis, each edge bunch in $G'$ must have at most $4$ edges.  Thus all edge bunches in $G$ have at most $4$ edges.

Now we deal with the case of $a=2$.  If there exists a divisor $D\geq 0$ of positive rank on $G'$ with $\deg(D)=\gon(G')$ that places $2$ chips on $w$, then $\gon(G)\leq \gon(G')$, as the same divisor has positive rank when considered on $G$.  Since the upper bound holds for $G'$, it holds for $G$ as well, since gonality has not increased but genus has.  Note that equality could only hold for $G$ if it holds for $G'$, and so in this case $G'$ (and thus $G$) must have all edge bunches of size at most $4$ by the inductive hypotehsis.

Now assume there does not exist such a divisor on $G'$.   Let $D\geq 0$ have positive rank on $G'$ with $D$ being $w$-reduced; since $D(w)>0$ and $D(w)< 2$, we know $D(w)=1$.  No chips may ever move from $w$ to its neighbors by the ripeness assumption, which means all neighbors of $w$ have at least one chip.  If any neighbor $u$ of $w$ had $|E(u,w)|=2$, then we could replace $D$ with $D-u+w$ to have an effective positive rank divisor of minimum degree placing $2$ chips on $w$; since no such divisor exists, we know that $|E(u,w)|\geq 3$ for all neighbors $u$ of $w$.

For each neighbor $u$ of $w$, let $B_u$ denote the connected component of $G'-w$; note that $B_u$ is a banana tree. Moreover, since no chips are moved between $w$ and the rest of the graph, we have that $\gon(G')=1+\sum_{u} \gon(B_u)$; and of course $\gon(G)\leq \gon(G')+1$.  By our inductive hypothesis,
\[\gon(G)\leq 1+\gon(G')= 2+\sum_{u} \gon(B_u)\leq 2+\sum_u \lfloor(g(B_u)-3)/2\rfloor \leq1+\left\lfloor\sum_u (g(B_u)+3)/2\right\rfloor.\]
The genus of $G$ is equal to the sum of the genera of the $B_u$ subgraphs, plus the number of edges connecting $w$ to its neighbors, plus the two edges between $v$ and $w$, minus $2$ for the additional two vertices $v$ and $w$; that is,
\[g(G)=1+\sum_{u}g(B_u)+\sum_u |E(u,w)|+2-2\geq 1+\sum_{u}(g(B_u)+3). \]
Thus we have
\[\gon(G)\leq 2+\left\lfloor\sum_u (g(B_u)+3)/2\right\rfloor\leq 2+\lfloor (g(G)-1)/2\rfloor=\lfloor g(G)+3/2\rfloor,\]
giving us the claimed inequality.  Finally, note that in order to have equality, each $B_u$ must have had $\gon(B_u)=\lfloor (g(B_u)+3)/2\rfloor$ for all $u$ implying by induction that each $B_u$ has no edge bunches of size larger than $4$; and each $|E(u,w)|$ must have been equal to $3$ with possibly one of size $4$ (due to the floor function).  Thus to have equality for $G$, a necessary condition is for each edge bunch to have at most $4$ edges.

To prove the final claim,  note that any bridgeless (that is, ripe) banana tree achieving the upper bound has each edge bunch of size $2$, $3$, or $4$.  By Lemma \ref{lemma:easy_upper_bound_gon}, we know $\gon(G)\leq\lcm(2,3,4)=12$.  In order to have $\gon(G)=\lfloor (g+3)/2\rfloor\leq 12$, a necessary condition is $g\leq 22$.  There are only finitely many ripe banana trees with genus at most $22$, completing the proof.
\end{proof}

This result suggests that banana trees, although useful for constructing examples of graphs with certain divisorial properties, do not tend to contain the full richness of Brill-Noether general graphs.  A promising direction for future research would be to look at metric versions of banana trees, or equivalently (by \cite{discrete_metric_different}) subdivisions of banana trees.  Indeed, metric versions of the banana path $B_{2,2,\ldots,2}$ provided the first known family of Brill-Noether general graphs in \cite{tropical_brill}.

\section{Computational results} \label{section:computationalresults}

We end this paper with some computational results on banana graphs. Throughout, we restrict ourselves to banana graphs with no edge bunches of size 1. We find gonalities of all ``small" banana path graphs and check for counterexamples to popular conjectures related to graph gonality.

Given Theorem \ref{theorem:polytime_gonality}, it is possible to efficiently investigate the gonalities of small banana path graphs with the use of computers. In particular, we exhaustively characterize the distribution of first gonalities for banana path graphs on $2$-$8$ vertices. Importantly, for a banana path graph on $n$ vertices, it suffices to study edge bunches of size at most $n$. This is because the gonality of any $n$ vertex graph is at most $n$ and an edge bunch of size $s > n$ partitions the instance into subproblems with edge bunches at most the number of vertices. 

\begin{table}[H]
\caption{Distribution of first gonalities for banana path graphs with $2$-$8$ vertices}
\centering
\begin{tabular}{c|rrrrrrr}
  \hline
  \multirow{2}{*}{Vertices} & \multicolumn{7}{c}{Gonality} \\
   & 2 & 3 & 4 & 5 & 6 & 7 & 8 \\ 
  \hline
  2 &   2 &   0 &   0 &   0 &   0 &   0 &   0 \\ 
  3 &   1 &   5 &   0 &   0 &   0 &   0 &   0 \\ 
  4 &   1 &   6 &  33 &   0 &   0 &   0 &   0 \\ 
  5 &   1 &   9 &  60 & 255 &   0 &   0 &   0 \\ 
  6 &   1 &  13 & 149 & 655 & 3178 &   0 &   0 \\ 
  7 &   1 &  17 & 311 & 1975 & 9803 & 46889 &   0 \\ 
  8 &   1 &  22 & 643 & 5311 & 36634 & 151517 & 856496 \\ 
  \hline
\end{tabular}
\end{table}

\noindent On these instances, we also checked Conjecture \ref{conj: high-order-gonality-conjectures} and and verified that no counterexample exists within the class of small banana path graphs. 

\begin{conjecture}[Question 1.3 of \cite{semigroup_gonality_sequences}] \label{conj: high-order-gonality-conjectures}
Let G be a graph with genus $g$ and the property that $\gon_2(G) = \gon_1(G) + 1$. Then, 
    \begin{enumerate}[label = (\arabic*)]
        \item $g = {\gon_1(G) \choose 2}$
        \item If $\gon_1(G) \geq 2$, then $\gon_3(G) = 2 \cdot \gon_1(G)$
    \end{enumerate}
\end{conjecture}

\noindent The code, results, and statistics can be found \href{https://github.com/KrishSingal/BananaGonality/}{here}.

\bibliographystyle{plain}

\end{document}